\documentclass{article}
\usepackage[utf8]{inputenc}
\usepackage{amsmath,amssymb,amsthm,graphicx,rotating,color}
\newcounter{mycount}

\newtheorem{theorem}[equation]{Theorem}

\newtheorem{definition}[equation]{Definition}
\newtheorem{remark}[equation]{Remark}
\DeclareMathSymbol{\leqslant}{\mathalpha}{AMSa}{"36}
\DeclareMathSymbol{\geqslant}{\mathalpha}{AMSa}{"3E}
\renewcommand{\le}{\;\leqslant\;}
\renewcommand{\ge}{\;\geqslant\;}

\def\O{\mathrm{O}}

\def\q{\quad}

\def\bx{\mathbf x}
\def\by{\mathbf y}

\def\bcdot{\,\cdot\,}
\def\qtext#1{\q\text{#1}\q}

\title{A binary deletion channel with a fixed number of deletions}
\author{
Benjamin Graham\\
{\small University of Warwick}\\
{\small \tt b.graham@warwick.ac.uk}\\
}

\begin{document}
\maketitle

\begin{abstract}
Suppose a binary string $\bx=x_1\dots x_n$ is being broadcast repeatedly over a faulty communication channel. Each time, the channel delivers a fixed number $m$ of the digits ($m<n$) with the lost digits chosen uniformly at random, and the order of the surviving digits preserved. How large does $m$ have to be to reconstruct the message?
\end{abstract}

\section{Introduction}
A {\em binary deletion channel} is a communication device that accepts a sequence of $n$ binary digits. Each digit is lost in transmission with probability $p$.  The order of the surviving digits is preserved, but the output does not indicate the original location of those digits. The number of digits in the output binary string thus follows the Binomial$(n,1-p)$ distribution.

There are two main questions associated with binary deletion channels; see \cite{Mitzenmacher-Survey} for a survey. First, can deletion channels be used to transmit information efficiently using some encoding scheme? Unlike binary symmetric channels and binary erasure channels, the exact information carrying capacity of the binary deletion channel is unknown. A lower bound on the information carrying capacity of the channel is $(1-p)/9$ \cite{DM}.


The other question concerns the reconstructability of the original message when it is transmitted across the deletion channel a number of times. This question is motivated in part by a biology problem---the task of sequencing DNA strands. Let $\bx\in\{0,1\}^n$ denote the message being transmitted. If $\bx$ is chosen uniformly at random, and if $p$ is sufficiently small, $\bx$ can be identified with high probability by looking at a polynomial number of samples from the deletion channel \cite{BKKM}. The difficulty of the reconstruction-from-samples problem is unknown for large $p$ and when $\bx$ may be chosen in an adversarial fashion.

To study the situation when $p$ is large and $\bx\in\{0,1\}^n$ is not necessarily chosen randomly, we will consider an alternative definition for the binary deletion channel. Rather than varying $p$, we will condition on the number of digits $m$ in the output. This is equivalent to choosing $m$ digits uniformly at random from the input digits; the value of $p$ no longer matters.
Our alternative definition is inspired by the difference between the two formulations $G(n,m)$ and $G(n,p)$ of the Erd\H{o}s-R\'enyi random graph: $G(n,m)$ has a fixed number of edges while the number of edges under $G(n.p)$ has the Binomial$(\binom{n}{2},p)$ distribution.

By fixing the number of deletions, we remove the option of sending $\bx$ through the deletion channel again and again until eventually none of the digits are deleted. It is therefore no longer trivial that $\bx$ can be reconstructed by studying the deletion channel output.

\section{The $(m,n)$-deletion channel}

Let $P(m,\bx)$ denote the probability distribution on $\{0,1\}^m$ generated by picking $1\le i_1<\dots<i_m\le n$ uniformly at random and returning the sequence $\by:=x_{i_1}\dots x_{i_m}$.
\begin{definition}
Let $R(m,n)$ denote the statement
\[
\text{for $\bx\in\{0,1\}^n$: the map $\bx\to P(m,\bx)$ is one-to-one}.
\]
\end{definition}
\noindent If $R(m,n)$ holds then $\bx$ can be determined by sampling repeatedly from $P(m,\bx)$. If not, there is a pair of length-$n$ binary strings that cannot be distinguished over an $(m,n)$-deletion channel, no matter how many times you sample.

\begin{definition}Let $N_m:=\sup \{n:R(m,n)\}$ denote the upper bound on the length of messages an $(m,\bcdot)$-deletion channel can convey.
\end{definition}
\begin{remark}
The first few terms of the sequence $(N_m)$ are
\[
N_1=1, \q N_2=3, \q N_3=6, \q N_4=11, \q N_5=15, \q N_6=29, \q \dots
\]
\end{remark}

\noindent The sequence appears to be growing exponentially. Equivalently, it seems that the elements of $\{0,1\}^n$ can be distinguished using a $(\O(\log n),n)$-deletion channel. This is perhaps unsurprising, given the high-dimensional nature of the $P(m,\bx)$ probability distributions.

\definecolor{digit}{rgb}{0.1,0.6,0.6}
\def\0{{\textcolor{digit}{0}}}
\def\1{{\textcolor{digit}{1}}}
Checking that $R(m,N_m)$ holds for $m\le 6$ was achieved by direct calculation. The difficulty of checking $R(m,n)$ grows very rapidly with $m$ and $n$.
For each $\bx\in\{0,1\}^n$, a $2^m$-dimensional vector representing $P(m,\bx)$ has to be calculated. The set of $2^n$ vectors then has to be searched for duplicates.

To demonstrate that $R(m,N_m+1)$ is false, we must provide a pair of binary sequence of length $N_m+1$ that produce a $P(m,\cdot)$ collision. Let $a_{\textcolor{digit}b}$ denote $a$ copies of $b$, i.e. $2_\0\equiv00$ and $3_\1\equiv111$.
For $m\in\{1,2,\dots,6\}$, examples of pairs of strings of length $N_m+1$ that produce a collision are:
\begin{align*}
   01\equiv1_\0 1_\1                             &\qtext{and}1_\1 1_\0 \equiv10,\\
   0110\equiv1_\0 2_\1 1_\0                           &\qtext{and}1_\1 2_\0 1_\1 \equiv1001,\\
1_\0 2_\1 3_\0 1_\1 &\qtext{and}1_\1 3_\0 2_\1 1_\0,\\
1_\0 2_\1 3_\0 1_\1 2_\0 2_\1 1_\0&\qtext{and}1_\1 3_\0 1_\1 1_\0 2_\1 3_\0 1_\1,\\
1_\0 2_\1 5_\0 3_\1 4_\0 1_\1&\qtext{and}1_\1 4_\0 3_\1 5_\0 2_\1 1_\0,\\
1_\0 2_\1 5_\0 3_\1 4_\0 1_\1 3_\0 3_\1 5_\0 2_\1 1_\0&\qtext{and}1_\1 4_\0 3_\1 5_\0 1_\1 1_\0 2_\1 5_\0 3_\1 4_\0 1_\1.
\end{align*}
We have not managed to find $N_7$ and $N_8$; without some theoretical advance they are computationally intractable. However, we have established the upper bounds $N_7<54$ and $N_8<106$. The bounds follow by checking that
\begin{align*}
P(7,&1_\0 2_\1 5_\0 3_\1 4_\0 1_\1 3_\0 3_\1 5_\0 2_\1 1_\0 2_\1 5_\0 2_\1 1_\0 1_\1 5_\0 3_\1 4_\0 1_\1)=\\
P(7,&1_\1 4_\0 3_\1 5_\0 1_\1 1_\0 2_\1 5_\0 2_\1 1_\0 2_\1 5_\0 3_\1 3_\0 1_\1 4_\0 3_\1 5_\0 2_\1 1_\0)
\end{align*}
and
\begin{align*}
P(8,&1_\0 2_\1 5_\0 3_\1 4_\0 1_\1 3_\0 3_\1 5_\0 2_\1 1_\0 2_\1 5_\0 2_\1 1_\0 1_\1 5_\0 3_\1 4_\0\ \rotatebox{270}{\hspace{-2mm}$\curvearrowright$}
\\
    &\ \rotatebox{90}{\hspace{1mm}$\curvearrowleft$}
1_\1 3_\0 3_\1 5_\0 1_\1 1_\0 2_\1 5_\0 2_\1 1_\0 2_\1 5_\0 3_\1 3_\0 1_\1 4_\0 3_\1 5_\0 2_\1 1_\0 )=\\
P(8,&1_\1 4_\0 3_\1 5_\0 1_\1 1_\0 2_\1 5_\0 2_\1 1_\0 2_\1 5_\0 3_\1 3_\0 1_\1 4_\0 3_\1 5_\0 1_\1 \ \rotatebox{270}{\hspace{-2mm}$\curvearrowright$} \\
    &\ \rotatebox{90}{\hspace{1mm}$\curvearrowleft$} 1_\0 2_\1 5_\0 3_\1 4_\0 1_\1 3_\0 3_\1 5_\0 2_\1 1_\0 2_\1 5_\0 2_\1 1_\0 1_\1 5_\0 3_\1 4_\0 1_\1).
\end{align*}
These binary strings were found by experimentally concatenating long substrings of the strings that form $P(6,30)$-collisions. Extrapolating from a dangerously small amount of data, these bounds appear to be the right order of magnitude.

We also have an upper bound on the whole sequence $(N_m)$. It is growing no more quickly than exponentially.
\begin{theorem}\label{theorem885226}
For some constant $C$, $N_m\le C^m$.
\end{theorem}
\begin{proof}
If $\bx\in\{0,1\}^n$, the probability distributions $P(m,\bx)$ is characterized by the probability of seeing each of the $2^m$ elements of $\{0,1\}^m$. The probability of seeing any particular element of $\{0,1\}^m$ is a multiple of $1/\binom{n}{m}$. The pigeonhole principle implies that $R(m,n)$ can only hold if
\begin{align}\label{nm}
\left(\binom{n}{m} + 1\right)^{2^m} \ge 2^n.
\end{align}
Set $n=C^m$, take logs and use the inequality $\binom{n}{m}+1 \le 2n^m$. Inequality \eqref{nm} holds only if
\[
2^m(1+ m^2 \log_2 C) \ge C^m.
\]
Clearly $R(m,C^m)$ cannot hold if $C$ is sufficiently large.
\end{proof}
Finding good lower bounds for $(N_m)$ seems much more difficult. We will only prove a linear bound.
\begin{theorem}\label{theorem598139}
$N_m\ge 2m-1$ for $m\ge 3$.
\end{theorem}
\begin{proof}[Proof of Theorem \ref{theorem598139}]
We will show that $\bx\in\{0,1\}^{2m-1}$ can be identified using $P(m,\bx)$. Note that for $j\le m$, we can deduce $P(j,\bx)$ from $P(m,\bx)$ using a second deletion channel which discards $m-j$ of its input digits.

Let $k$ denote the number of ones in $\bx$: $k$ is simply $2m-1$ times the probability of $1$ under $P(1,\bx)$. By symmetry we can assume $k<m$.
The string $\bx$ can be written as $k+1$ runs of zeros separated by the $k$ ones. Let $i^{(0)},i^{(1)},\dots,i^{(k)}\ge 0$ denote the length of the runs of zeros, i.e.
\[
\bx=  i^{(0)}_\0 1_\1 i^{(1)}_\0 1_\1 \dots 1_\1 i^{(k)}_\0.
\]
The $i^{(j)}$ are determined by the probabilities under $P(k+1,\bx)$ of the $k+1$ strings containing a single zero and $k$ ones:
\[
\text{ the probability of } j_\1 1_\0 (k-j)_\0 \text{ under } P(k+1,\bx) \text{ is } \frac{i^{(j)}}{\binom{2m-1}{k+1}}.\qedhere
\]
\end{proof}

\section{Conclusions}
We have introduced an alternative model for deletion channels; it has a non-trivial reconstructability problem. We have explored the space of `hardest-to-transmit' binary sequences to find $N_m$ for $m$ small. We have also found bounds on $N_m$ for general $m$.

We conjecture that $(N_m)$ grows exponentially.



\end{document}